\documentclass[12pt]{amsart}
\usepackage[margin=1.2in]{geometry}
\usepackage{graphicx,latexsym}
\usepackage{comment}
\usepackage{tikz}
\usepackage{amsfonts, amssymb, amsmath, amsthm, bm, hyperref}
\usepackage{tikz}
\usetikzlibrary{matrix,arrows}

\newcommand{\N}{\mathbb{N}}
\newcommand{\Z}{\mathbb{Z}}
\newcommand{\R}{\mathbb{R}}
\newcommand{\C}{\mathbb{C}}

\newcommand{\dx}{{\rm d}x }

\newcommand{\dt}{{\rm d}t }

\newcommand{\csn}{\operatorname{csn}}
\newcommand{\loc}{\operatorname{loc}}
\newcommand{\beq}{\begin{eqnarray}}
\newcommand{\eeq}{\end{eqnarray}}

\newcommand{\beqs}{\begin{eqnarray*}}
\newcommand{\eeqs}{\end{eqnarray*}}

\newtheorem{theorem}{Theorem}[section]
\newtheorem{proposition}[theorem]{Proposition}
\newtheorem{lemma}[theorem]{Lemma}
\newtheorem{corollary}[theorem]{Corollary}

\theoremstyle{definition}
\newtheorem{definition}[theorem]{Definition}
\newtheorem{example}[theorem]{Example}
\newtheorem{problem}[theorem]{Problem}

\theoremstyle{remark}
\newtheorem{remark}[theorem]{Remark}

\numberwithin{equation}{section}
\title[Quasinormable $C_0$-groups and the spaces of type $\mathcal{D}_E$]{Quasinormable $C_0$-groups and translation-invariant Fr\'echet  spaces of type $\mathcal{D}_E$}

\author[A. Debrouwere]{Andreas Debrouwere}
\address{Department of Mathematics: Analysis, Logic and Discrete Mathematics, Ghent University, Krijgslaan 281, 9000 Gent, Belgium}
\email{Andreas.Debrouwere@UGent.be}
\thanks{The author is supported by  FWO-Vlaanderen, via the postdoctoral grant 12T0519N}

\begin{document}
\begin{abstract}
Let $E$ be a locally convex Hausdorff space satisfying the convex compact property and let $(T_x)_{x \in \R^d}$ be a locally equicontinuous $C_0$-group of linear continuous operators on $E$. In this article, we show that if $E$ is quasinormable, then the space  of smooth vectors in $E$ associated to $(T_x)_{x \in \R^d}$ is also quasinormable. In particular, we obtain that the space of smooth vectors associated to a $C_0$-group on a Banach space is always quasinormable. As an application, we show that the translation-invariant Fr\'echet spaces of smooth functions of type $\mathcal{D}_E$ \cite{D-Pi-V} are quasinormable, thereby settling the question posed in \cite[Remark 7]{D-Pi-V}. Furthermore, we show that  $\mathcal{D}_E$ is not Montel if $E$ is a solid translation-invariant Banach space of distributions \cite{F-G}. This answers the question posed in \cite[Remark 6]{D-Pi-V} for the class of solid translation-invariant Banach spaces of distributions.
\end{abstract}

\maketitle

\section{Introduction}
The class of quasinormable locally convex spaces was introduced by Grothendieck in \cite{Grothendieck} (see also \cite{M-V}) and plays an important role in the lifting and splitting theory for Fr\'echet spaces \cite{B-D, Valdivia, M-V-0,Vogt-87}. Most spaces appearing in analysis are quasinormable, e.g., all Banach and Schwartz spaces are quasinormable while examples of quasinormable spaces which are neither Banach nor Montel are given by the space $C(X)$ of continuous functions on a non-compact completely regular Hausdorff space $X$, the space $C^n(\Omega)$, $n \in \N$, of $n$-times differentiable functions on an open subset $\Omega$ of $\R^d$ and the spaces $\mathcal{D}_{L^p}(\R^d)$ ($1 \leq p < \infty$), $\mathcal{B}(\R^d)$ and $\dot{\mathcal{B}}(\R^d)$ of Schwartz. Furthermore, the quasinormability of various weighted function spaces has been characterized in terms of the defining weights; see \cite{B-M-S-1, Valdivia-81, Vogt-84} for K\"othe echelon spaces, \cite{B-E, B-M} for weighted spaces of continuous functions and \cite{Wolf-1,Wolf-2, Wolf-3} for weighted spaces of holomorphic functions. 

In the first part of this article, we study quasinormability in the context of $C_0$-groups. Namely, let $E$ be a locally convex Hausdorff space satisfying the convex compact property and let $(T_x)_{x \in \R^d}$ be a locally equicontinuous $C_0$-group of continuous linear operators on $E$; see Sections \ref{sect-prelim} and \ref{sect-conv} for the definition of these notions. Our main result asserts that if $E$ is quasinormable, then the space  of smooth vectors in $E$ associated to $(T_x)_{x \in \R^d}$, endowed with its natural locally convex topology, is also quasinormable. Since every Banach space is quasinormable and every $C_0$-group on a Banach space is locally equicontinuous, we obtain particularly that the space of smooth vectors associated to a $C_0$-group on a Banach space is always quasinormable.  

The second part of this article is devoted to the study of the linear topological properties of the translation-invariant Fr\'echet spaces of type $\mathcal{D}_E$. Translation-invariant Banach spaces  of distributions (TIBD) $E$ and the associated test function spaces $\mathcal{D}_E$ were introduced in \cite{D-Pi-V} and are natural generalizations of the spaces  ${L^p}(\R^d)$ ($1\leq p < \infty$) and $C_0(\R^d)$, and $\mathcal{D}_{L^p}(\R^d)$ and $\dot{\mathcal{B}}(\R^d)$, respectively; we refer to Definitions \ref{Def-1} and \ref{Def-2} below for the precise definition of these notions. Firstly, we show that the space $\mathcal{D}_E$ is quasinormable  for any TIBD $E$. Since every quasinormable Fr\'echet space is distinguished, this  settles the question posed in  \cite[Remark 7]{D-Pi-V}. To this end, we show that the space of smooth vectors associated to the translation group on $E$ coincides with  $\mathcal{D}_E$, whence the quasinormability of $\mathcal{D}_E$ follows immediately from the result obtained in the first part of this article. In particular, we obtain  direct proofs of the quasinormability of $\mathcal{D}_{L^p}(\R^d)$ and $\dot{\mathcal{B}}(\R^d)$; this also follows from the sequence space representations $\mathcal{D}_{L^p}(\R^d) \cong s \widehat{\otimes} l_p$ and $\dot{\mathcal{B}}(\R^d) \cong s \widehat{\otimes} c_0$  \cite[Thm.\ 3.2]{Vogt-83}, while a proof of the quasinormability of $\dot{\mathcal{B}}(\R^d)$ based on duality arguments and the theory of $(LB)$-spaces is given in \cite{D-D}. Secondly,  we show that  $\mathcal{D}_E$ is not Montel if $E$ is a solid TIBD \cite{F-G}, thereby generalizing the well-known fact that the spaces $\mathcal{D}_{L^p}(\R^d)$ and $\dot{\mathcal{B}}(\R^d)$ are not Montel; see Definition \ref{Def-1} below  for the definition of a solid TIBD. This answers the question posed in  \cite[Remark 6]{D-Pi-V} for the class of solid TIBD. We believe that $\mathcal{D}_E$ is never Montel but we were not able to show this for general TIBD $E$.

This paper is organized as follows. In the preliminary Section \ref{sect-prelim}, we fix the notation and introduce the class of $C_0$-groups to be considered.
Our main result concerning $C_0$-groups is shown in Section \ref{sect-quasi}. Its proof is based on a quantitative convolution approximation identity for differentiable vectors associated to a  $C_0$-group, which is proven separately in the auxiliary Section \ref{sect-conv}.  Finally, in Section \ref{TIB}, we  present our results about the spaces  $\mathcal{D}_E$.
\section{Preliminaries}\label{sect-prelim}
In this section, we fix the notation, introduce $C_0$-groups (parameterized by $\R^d$) on general locally convex Hausdorff spaces (cf.\ \cite{Komura}) and define the space of smooth vectors associated to a given $C_0$-group.

We set $\N = \{0,1,2 \ldots \}$. Let $E$ be a Hausdorff locally convex space (from now on abbreviated as lcHs). We denote by $\mathcal{U}_0(E)$ the set of all neighbourhoods of $0$ in $E$, by $\mathcal{B}(E)$ the set of all bounded sets in $E$ and by $\csn(E)$ the set of all continuous seminorms on $E$. Furthermore, $\mathcal{L}(E)$ stands for the space of all continuous linear operators from $E$ into itself. We always endow $\mathcal{L}(E)$ with the strong operator topology, that is, the topology generated by the system of seminorms $\{p_e \, | \, p \in \csn(E), e \in E \}$, where
$$
p_{e}(T) := p(Te), \qquad T \in \mathcal{L}(E),  p \in \csn(E), e \in E.
$$
 The space of compactly supported continuous functions on $\R^d$ is denoted by $C_c(\R^d)$. Furthermore, we write $C^n_c(\R^d) = C^n(\R^d) \cap C_c(\R^d)$, $n\in \N$, and $\mathcal{D}(\R^d) = C^\infty(\R^d) \cap C_c(\R^d)$. 

Let $E$ be a lcHs. A family  $(T_x)_{x \in \R^d} \subset \mathcal{L}(E)$ is said to be a $C_0$-group on  $E$ if the following conditions are satisfied
\begin{itemize}
\item[$(i)$] $T_0 = \operatorname{id}$.
\item[$(ii)$] $T_{x + y} = T_{x} \circ T_{y}$ for all $x,y \in \R^d$.
\item[$(iii)$] $\lim_{x \rightarrow 0} T_x e = e$ for all $e \in E$.
\end{itemize}
These conditions imply that the mapping
\begin{equation}
\R^d \rightarrow \mathcal{L}(E): x \rightarrow T_x
\label{def}
\end{equation}
is  continuous. The $C_0$-group $(T_x)_{x \in \R^d}$ is called locally equicontinuous if, for each compact subset $K$ of $\R^d$, the set $\{ T_x \, | \, x \in K\}$ is equicontinuous. 
If $E$ is barrelled, then every $C_0$-group on $E$ is locally equicontinuous, as follows from the continuity of the mapping \eqref{def} and the Banach-Steinhaus theorem.

Let $E$ be a lcHs and let  $(T_x)_{x \in \R^d}$ be a $C_0$-group on $E$.  The orbit $\gamma_e$ of a vector $e \in E$ is defined  as the mapping $\gamma_e: \R^d \rightarrow E: x \rightarrow T_xe$. The continuity of the mapping \eqref{def} implies that $\gamma_e \in C(\R^d;E)$.
A vector $e \in E$ is called $n$-times differentiable, $n \in \N \, \cup \, \{\infty\}$,  if $\gamma_e \in C^n(\R^d;E)$. The space of all $n$-times differentiable vectors in $E$ is denoted by $E^n$. Fix $n \in \N$. We set
$$
e^{(\alpha)} := \partial^\alpha \gamma_e(0), \qquad e \in E^n, |\alpha| \leq n.
$$
Notice that
$$
\gamma_{e^{(\alpha)}} = \partial^\alpha \gamma_e, \qquad e \in E^n, |\alpha| \leq n.
$$
We endow $E^n$  with the initial topology with respect to the mapping 
$$
E^n \rightarrow \prod_{|\alpha| \leq n} E: e \rightarrow (e^{(\alpha)})_{|\alpha| \leq n},
$$
which means that the topology of $E^n$ is generated by the system of seminorms $\{p_n \, | \, p \in \csn(E)\}$, where
$$
p_n(e) := \max_{|\alpha| \leq n} p(e^{(\alpha)}), \qquad e \in E^n, p \in \csn(E).
$$
Similarly, we endow $E^\infty$ with with the initial topology with respect to the mapping 
$$
E^\infty \rightarrow \prod_{\alpha \in \N^d} E: e \rightarrow (e^{(\alpha)})_{\alpha \in \N^d},
$$
which means that the topology of $E^\infty$ is generated by the system of seminorms $\{p_n \, | \, p \in \csn(E), n \in \N \}$. Let $n \in \N \, \cup \, \{\infty\}$. If $(T_x)_{x \in \R^d}$ is locally equicontinuous, the mapping
$$
E^n \rightarrow C^n(\R^d;E): e \rightarrow \gamma_e
$$
is a topological embedding. In particular, $E^n$ is a Fr\'echet space if $E$ is so.
\section{Convolution with respect to a $C_0$-group}\label{sect-conv}
Let $E$ be  a lcHs and let  $(T_x)_{x \in \R^d}$ be a  $C_0$-group on $E$. In this section, we define and give some basic properties of the convolution product $\varphi \ast_T e$, $\varphi \in C_c(\R^d)$, $e \in E$, with respect to $(T_x)_{x \in \R^d}$. Most importantly, we provide a quantitative convolution approximation identity for differentiable vectors in $E$. This result shall be the crux of the proof of our main theorem given in the next section.  The results in this section are probably well-known, as they are straightforward analogues of basic results from the theory of one-parameter $C_0$-semigroups \cite{H-P}, but we include them here with proofs for the sake of completeness.

We start with a brief discussion about vector-valued integration \cite{Rudin}. A lcHs $E$ is said to satisfy the convex compactness property, for short $(cc)$, if the closed absolutely convex hull of
every compact subset of $E$ is again compact. Every quasi-complete lcHs satisfies $(cc)$, while sequential completeness and $(cc)$ are incomparable, that is, there are lcHs which are sequentially complete but do not satisfy the $(cc)$ and vice versa \cite[p.\ 1421]{Qui}. The property $(cc)$ for $E$ is closely connected to the existence of $E$-valued weak integrals, as we now proceed to explain. A function $\Phi: \R^d \rightarrow E$ is called scalarly integrable if $\langle e', \Phi(\cdot) \rangle \in L^1(\R^d)$ for all $e' \in E'$. In such a case, the mapping
$$
e' \rightarrow \int_{\R^d} \langle e', \Phi(x) \rangle \dx
$$
defines an element of the algebraic dual of $E'$, which we denote by $\int_{\R^d}  \Phi(x) \dx$.  The function $\Phi$ is said to be weakly integrable in $E$ if $\int_{\R^d} \Phi(x) \dx \in E$.
If $E$ satisfies $(cc)$, every $E$-valued compactly supported continuous function $\Phi$ on $\R^d$ is weakly integrable in $E$ \cite[Thm.\ 3.27]{Rudin}. Moreover, we have that
$$
p \left (\int_{\R^d} \Phi(x) \dx \right) \leq \int_{\R^d} p(\Phi(x)) \dx, \qquad  p \in \csn(E),
$$
and
$$
\int_{\R^d}  (S \circ \Phi)(x) \dx = S \left (\int_{\R^d} \Phi(x) \dx \right), \qquad S \in \mathcal{L}(E).
$$

Let $E$ be a lcHs satisfying $(cc)$ and let  $(T_x)_{x \in \R^d}$ be a $C_0$-group on $E$. Given $\varphi \in C_c(\R^d)$ and $e \in E$, we define
$$
\varphi \ast_T e : = \int_{\R^d} \varphi(x) T_{-x}e \dx =  \int_{\R^d} \varphi(x) \gamma_e(-x) \dx \in E.
$$
Notice that
$$
\gamma_{\varphi \ast_T e}(x) =  \int_{\R^d} \varphi(t) \gamma_e(x - t)\dt = \int_{\R^d} \varphi(x - t) \gamma_e(t)\dt, \qquad  \varphi \in C_c(\R^d),e \in E.
$$
We start with the following basic lemma.
\begin{lemma}\label{switch-conv}  Let $E$ be a lcHs satisfying $(cc)$, let  $(T_x)_{x \in \R^d}$ be a locally equicontinuous $C_0$-group on $E$ and let $n \in \N$.
\begin{itemize}
\item[$(i)$] For all $\varphi \in C^n_c(\R^d)$ and $e \in E$ it holds that  $\varphi \ast_T e \in E^n$ and  $(\varphi \ast_T e)^{(\alpha)} = \partial^\alpha \varphi \ast_T e$ for all $|\alpha| \leq n$. Moreover,
for all $p \in \csn(E)$ and $K \subset \R^d$ compact there is $q \in \csn(E)$ such that
$$
p_n(\varphi \ast_T  e) \leq \max_{|\alpha| \leq n } \| \partial^\alpha \varphi\|_{L^1} q(e), \qquad \varphi \in C^n_c(\R^d) \mbox{with $\operatorname{supp} \varphi \subseteq K$}, e \in E.
$$
\item[$(ii)$] For all $\varphi \in C_c(\R^d)$ and $e \in E^n$ it holds that $\varphi \ast_T e \in E^n$ and $(\varphi \ast_T e)^{(\alpha)} = \varphi \ast_T e^{(\alpha)}$ for all $|\alpha| \leq n$. Moreover, for all $p \in \csn(E)$ and $K \subset \R^d$ compact there is $q \in \csn(E)$ such that
$$
p_n(\varphi \ast_T  e) \leq  \| \varphi\|_{L^1} q_n(e), \qquad \varphi \in C_c(\R^d)  \mbox{with $\operatorname{supp} \varphi \subseteq K$}, e \in E^n.
$$
\end{itemize}
\end{lemma}
\begin{proof}
We only show $(ii)$ as the proof of $(i)$ goes along the same lines. It suffices to consider the case $n = 1$, the general case then follows by induction. Furthermore,  we only need  to show that $\varphi \ast_T e \in E^1$ and $(\varphi \ast_T e)^{(e_j)} = \varphi \ast_T e^{(e_j)}$ for all $1\leq j \leq d$, $\varphi \in C_c(\R^d)$ and $e \in E^1$. Once this relation is established, the statement about the seminorms is a consequence of the fact that $(T_x)_{x \in \R^d}$ is locally equicontinuous. Fix $x \in \R^d$. Then, for all $p \in \csn(E)$ and $h \in \R \, \backslash \, \{0\}$,
\begin{align*}
& p \left ( \frac{\gamma_{\varphi \ast_T e}(x + he_j) - \gamma_{\varphi \ast_T e}(x) }{h} -  \gamma_{\varphi \ast_T e^{(e_j)}}(x) \right) \\
&= p \left ( \int_{\R^d} \varphi(t) \left(\frac{\gamma_{e}(x + he_j -t) - \gamma_{e}(x-t) }{h} -  \gamma_e^{(e_j)}(x-t)\right) \dt \right)\\
&\leq  \int_{\R^d} |\varphi(t)| p\left(\frac{\gamma_{e}(x + he_j -t) - \gamma_{e}(x-t) }{h} -  \gamma_e^{(e_j)}(x-t)\right) \dt. 
\end{align*}
The result will therefore follow from the Lebesgue dominated convergence theorem if the set 
$$
B = \left\{\frac{\gamma_{e}(x + he_j -t) - \gamma_{e}(x-t) }{h} -  \gamma_e^{(e_j)}(x-t) \, | \, t \in \operatorname{supp} \varphi, 0 < |h| \leq 1 \right \}
$$
is bounded in $E$. By Mackey's theorem, it suffices to show that $B$ is weakly bounded in $E$. Let $e' \in E'$ be arbitrary and set $f_{e'} = \langle e', \gamma_e(\,\cdot \,) \rangle \in C^1(\R^d)$. Then,
\begin{align*}
&\left | \left \langle e', \frac{\gamma_{e}(x + he_j -t) - \gamma_{e}(x-t) }{h} -  \gamma_e^{(e_j)}(x-t) \right \rangle \right | \\
&= \left |\frac{f_{e'}(x+he_j - t) - f_{e'}(x-t)}{h} - f_{e'}^{(e_j)}(x-t) \right| \\
&\leq 2 \max\{|f_{e'}^{(e_j)}(x + ke_j -u) | \, | \, u \in \operatorname{supp} \varphi, |k| \leq 1 \}
\end{align*}
for all $t \in \operatorname{supp} \varphi$ and $0 < |h| \leq 1$.
\end{proof}
Next, we study the approximation of vectors in $E$ via convolution.
\begin{lemma}\label{approx}
Let $E$ be a lcHs satisfying $(cc)$ and let  $(T_x)_{x \in \R^d}$ be a locally equicontinuous $C_0$-group on $E$. Let $\varphi \in C_c(\R^d)$ with $\int_{\R^d} \varphi(x) \dx = 1$ and set $ \varphi_r = r^{-d} \varphi( \, \cdot \, / r)$  for $r > 0$. Then,
\begin{equation}
\lim_{r \to 0^+} \varphi_r \ast_T e = e 
\label{limit}
\end{equation}
for all $e \in E$.
\end{lemma}
\begin{proof}
Let $p \in \csn(E)$ be arbitrary. Then,
\begin{align*}
\lim_{r \to 0^+}p( \varphi_r \ast_T e - e) &= \lim_{r \to 0^+} p\left(\int_{\R^d} \varphi(x)(T_{-rx} e - e) \dx \right) \\
&\leq \lim_{r \to 0^+}\int_{\R^d} |\varphi(x)| p(T_{-rx} e - e) \dx = 0,
\end{align*}
where the last step follows from the Lebesgue dominated convergence theorem and the fact that $T_t e \rightarrow e$ in $E$ as $t \rightarrow 0$.
\end{proof}
 We now show that the rate of convergence in the approximation \eqref{limit} is of type $O(r)$ for $e \in E^1$.
\begin{proposition} \label{controlled-approx}
Let $E$ be a lcHs satisfying $(cc)$ and let  $(T_x)_{x \in \R^d}$ be a locally equicontinuous $C_0$-group on $E$. Let $\varphi \in C_c(\R^d)$ with $\int_{\R^d} \varphi(x) \dx = 1$ and set $ \varphi_r = r^{-d} \varphi( \, \cdot \, / r)$  for $r > 0$. Then, for all $p \in \csn(E)$ there  is  $q \in \csn(E)$  such that
$$
p(\varphi_r \ast_T e - e) \leq r q_1(e), \qquad e \in E^1, r \leq 1.
$$
\end{proposition}

Proposition \ref{controlled-approx} is a consequence of the next lemma (cf.\ the proof of Lemma \ref{approx}).
\begin{lemma}
Let $R > 0$. Then, for all $p \in \csn(E)$ there is $q \in \csn(E)$ such that
$$
p(T_xe - e) \leq |x| q_1(e), \qquad e \in E^1, |x| \leq R.
$$
\end{lemma}
\begin{proof}
Since $(T_x)_{x \in \R^d}$ is locally equicontinuous, there is $q' \in \csn(E)$ such that $p(T_xe) \leq q'(e)$ for all $e \in E$ and $|x| \leq R$. Set $U = \{e \in E \, | \, p(e) \leq 1 \}$ and denote by $U^\circ$ the polar set of $U$ in $E'$. As before, we write $f_{e'} = \langle e', \gamma_e(\,\cdot \,) \rangle \in C^1(\R^d)$ for $e \in E^1$ and  $e' \in E'$. Then,
\begin{align*}
&p(T_xe-e) = \sup_{e' \in U^\circ} | f_{e'}(x) - f_{e'}(0) | \leq |x| \sup_{e' \in U^\circ}\sum_{j=1}^d \int_0^1 |f_{e'}^{(e_j)}(tx)| \dt \\
&=|x| \sup_{e' \in U^\circ} \sum_{j=1}^d \int_0^1 |\langle e', \gamma_{e^{(e_j)}}(tx) \rangle  | \dt \leq |x| \sum_{j=1}^d \int_0^1 p(T_{tx}e^{(e_j)})\dt \leq |x|dq'_1(e)
\end{align*}
for all $e \in E^1$ and $|x| \leq R$. Hence, the statement holds with $q = dq' \in \csn(E)$.
\end{proof}

\section{Quasinormability of $E^\infty$}\label{sect-quasi}
This section contains the main results of the first part of this article. Given a lcHs $E$ satisfying $(cc)$ and a locally equicontinuous $C_0$-group $(T_x)_{x \in \R^d}$ on $E$,  it is shown that the space $E^\infty$ is quasinormable if $E$ is so. In case $E$ is a Fr\'echet space, we also prove that $E^\infty$ satisfies the linear topological invariant $(\Omega)$ of Vogt \cite{M-V} if $E$ does so. Finally, by using a result from \cite{B-E}, we give an example that shows that the converse of this theorem does not hold in general.

A lcHs $E$ is said to be quasinormable \cite[p.\ 313]{M-V} if
$$
\forall U \in \mathcal{U}_0(E) \, \exists V  \in \mathcal{U}_0(E) \, \forall \varepsilon > 0 \, \exists B \in \mathcal{B}(E) \, : \, V \subset \varepsilon U + B,
$$
or, equivalently,
\begin{gather*}
\forall p \in \csn(E) \, \exists q \in \csn(E)  \,  \forall \varepsilon > 0 \, \exists B \in \mathcal{B}(E)\, : \\
 \forall x \in E \mbox{ with } q(x) \leq 1 \, \exists y \in B \mbox{ such that } p(x -y) \leq \varepsilon.
\end{gather*}
Obviously, every Banach space is quasinormable, while a lcHs is Schwartz if and only if it is quasinormable and semi-Montel. Furthermore, every quasinormable Fr\'echet space is distinguished \cite[Cor.\ 26.19]{M-V}.

We are ready to prove the main result of this section.
\begin{theorem} \label{main}
Let $E$ be a lcHs satisfying $(cc)$ and let $(T_x)_{x \in \R^d}$ be a locally equicontinuous $C_0$-group on $E$. Then, $E^\infty$ is quasinormable if $E$ is so.
\end{theorem}
\begin{proof}
It suffices to show that
\begin{gather*}
\forall p \in \csn(E) \, \forall n \in \N \, \exists q \in \csn(E) \, \exists m \in \N \,  \forall \varepsilon > 0 \, \exists B \in \mathcal{B}(E^\infty)\, : \\ 
 \forall e \in E^\infty \mbox{ with } q_m(e) \leq 1 \, \exists x \in B \mbox{ such that } p_n(e -x) \leq \varepsilon.
\end{gather*}
Let $p \in \csn(E)$ and $n \in \N$ be arbitrary. Pick $\varphi \in \mathcal{D}(\R^d)$ with $\int_{\R^d} \varphi(x) \dx = 1$ and set $ \varphi_r = r^{-d} \varphi( \, \cdot \, / r)$  for $r > 0$. Lemma \ref{switch-conv}$(i)$ implies that there is $p'' \in \csn(E)$ such that
\begin{equation}
p_n( \varphi_r \ast_T e)\leq  \max_{|\alpha| \leq n } \| \partial^\alpha \varphi_r\|_{L^1} p''(e) \leq \frac{1}{r^n}p'(e), \qquad e \in E, r \leq 1,
\label{bound-proof}
\end{equation}
where $p' = (\max_{|\alpha| \leq n } \| \partial^\alpha \varphi\|_{L^1})p'' \in \csn(E)$. Lemma \ref{switch-conv}$(ii)$ and Proposition \ref{controlled-approx}  yield that there is $q' \in \csn(E)$ such that
\begin{equation}
p_n(\varphi_r \ast_T e - e) \leq r q'_{n+1}(e), \qquad e \in E^{n+1}, r \leq 1.
\label{bound-proof-2}
\end{equation}
Since $E$ is quasinormable, there is $s \in \csn(E)$ such that
\begin{equation}
\forall \delta > 0 \, \exists A \in \mathcal{B}(E) \,: \, \forall e \in E \mbox{ with } s(e) \leq 1 \, \exists y \in A \mbox{ such that } p'(e -y) \leq \delta.
\label{known}
\end{equation}
Set $q = \max\{s,q'\} \in \csn(E)$ and $m = n + 1$. Let $\varepsilon > 0$ be arbitrary. We may assume without loss of generality that $\varepsilon \leq 1$. Choose $A \in  \mathcal{B}(E)$ according to \eqref{known} with $\delta = (\varepsilon/2)^{n+1}$.
Set $B = \{ \varphi_{\varepsilon/2} \ast_T y \, | \, y \in A\} \in  \mathcal{B}(E^\infty)$. Now let $e \in E^\infty$ with $q_m(e) \leq 1$ be arbitrary. Since $s(e) \leq q_m(e) \leq 1$ there is $y \in A$ such that $p'(e-y) \leq  (\varepsilon/2)^{n+1}$. Set $x = \varphi_{\varepsilon/2} \ast_T y \in B$. Then, by \eqref{bound-proof} and \eqref{bound-proof-2}, we have that
$$
p_n(e - x) \leq p_n(e - \varphi_{\varepsilon/2} \ast_T e) + p_n( \varphi_{\varepsilon/2} \ast_T e - \varphi_{\varepsilon/2} \ast_T y) \leq \varepsilon.
$$
\end{proof}
Every quasinormable lcHs $E$ clearly satisfies
$$
\forall  U \in \mathcal{U}_0(E) \, \exists V \in \mathcal{U}_0(E) \, \forall W \in \mathcal{U}_0(E) \, \forall \varepsilon > 0 \, \exists R > 0 \, : \, V \subset \varepsilon U + R W.
$$
If $E$ is a Fr\'echet space, this condition is also sufficient for $E$ to be quasinormable \cite[Lemma 26.14]{M-V}. On the other hand, a Fr\'echet space $E$ is said to satisfy condition $(\Omega)$ \cite[p.\ 367]{M-V} if 
$$
\forall  U \in \mathcal{U}_0(E) \, \exists V \in \mathcal{U}_0(E) \, \forall W \in \mathcal{U}_0(E) \, \exists C,s > 0 \,\forall \varepsilon > 0 \, : \, V \subset \varepsilon U + \frac{C}{\varepsilon^s} W.
$$
Hence, in the class of Fr\'echet spaces, $(\Omega)$ may be considered as a quantified version of quasinormability. A careful inspection of the proof of Theorem \ref{main} shows that the following result holds.
\begin{theorem}
Let $E$ be a Fr\'echet space and let $(T_x)_{x \in \R^d}$ be a $C_0$-group on $E$. Then, $E^\infty$ satisfies $(\Omega)$ if $E$ does so.
\end{theorem}
Since every Banach space satisfies $(\Omega)$, we obtain the following corollary.
\begin{corollary}\label{Banach-omega}
Let $E$ be a Banach space and let $(T_x)_{x \in \R^d}$ be a $C_0$-group on $E$. Then, $E^\infty$ satisfies $(\Omega)$. In particular, $E^\infty$ is quasinormable and distinguished.
\end{corollary}
Finally, we show that the converse of Theorem \ref{main} does not hold in general. We need some preparation. Let $X$ be a completely regular Hausdorff topological space and let $V$ be a (positive continuous) Nachbin family on $X$ \cite{B-M-S}, i.e., a family consisting of positive continuous functions on $X$ such that for all $v_1, v_2 \in V$ and $\lambda > 0$ there is $w \in V$ such that $\lambda \max\{v_1,v_2\} \leq w$. The associated weighted Nachbin space $VC_0(X)$ consists of all $f \in C(X)$ such that $vf$ vanishes at infinity for all $v \in V$; endowed with the topology generated by the system of seminorms $\{ p_v \, | \, v \in V \}$, where
$$
p_v(f) := \sup_{t \in X} v(t) |f(t)|, \qquad f \in VC_0(X), v \in V,
$$
it becomes a complete lcHs. In particular, $VC_0(X)$ satisfies $(cc)$. Bastin and Ernst \cite{B-E}  characterized the quasinormability of $VC_0(X)$ in the following way; see also \cite{B-M} for countable Nachbin families.

\begin{proposition}\cite[Prop.\ 2]{B-E} \label{Bastin}
Let $X$ be a completely regular Hausdorff topological space and let $V$ be a Nachbin family on $X$. Then, $VC_0(X)$ is quasinormable if and only if $V$ is regularly increasing, i.e., 
\begin{gather*}
\qquad \forall v \in V \, \exists u \in V \mbox { with } u \geq v \, \forall \varepsilon > 0 \, \forall w \in V \mbox { with } w \geq u \, \exists \delta > 0 \, \forall t \in X \,: \, 
\\ v(t) \geq \varepsilon u(t) \Longrightarrow u(t) \geq \delta w(t). 
\end{gather*}
\end{proposition}
We now give an example of a lcHs $E$ and a locally equicontinuous $C_0$-group on $E$ such that $E^\infty$ is quasinormable but $E$ is not; we were inspired by \cite[Example 3.1]{A-B-R}.
\begin{example}\label{neg-examle}
Let $X$ be a non-compact normal Hausdorff topological space and let $V$ be a  Nachbin family on $X$ that is not regularly increasing. Consider a positive unbounded continuous function $g$ on $X$; such a function can be constructed by using the non-compactness of $X$ and Tietze's extension theorem. One can readily check that the family of continuous linear operators $(T_x)_{x \in \R}$ given by
$$
T_x : VC_0(X) \rightarrow VC_0(X) : f \rightarrow e^{ixg}f, \qquad x \in \R,
$$
is a locally equicontinuous $C_0$-group on $E = VC_0(X)$ with $E^\infty = WC_0(X)$, where $W = \{ g^n v \, | \, n \in \N, v \in V\}$. Notice that $W$ is a regularly increasing Nachbin family. Hence, Proposition \ref{Bastin} yields that $E^\infty$ is quasinormable but $E$ is not.
\end{example}
\section{Translation-invariant Banach spaces of distributions}\label{TIB} 
In this final section, we discuss the linear topological properties of the translation-invariant Fr\'echet spaces of type $\mathcal{D}_E$. 
From now on, $T_x$, $x \in \R^d$, will always denote the translation operator on the space $\mathcal{D}'(\R^d)$ of distributions, i.e., $T_x f = f( \, \cdot  + x)$, $f \in \mathcal{D}'(\R^d)$. Furthermore,  the symbol ``$\hookrightarrow$" stands for dense continuous inclusion. 

We start by introducing translation-invariant Banach spaces of distributions; our definition is slightly more general than the one given in \cite{D-Pi-V} (see Remark \ref{general-def} below). 
\begin{definition}\label{Def-1}
A Banach space $E$ is said to be a translation-invariant Banach space of distributions (TIBD) if the following two conditions are satisfied
\begin{itemize}
\item[$(i)$] $\mathcal{D}(\R^d) \hookrightarrow E \hookrightarrow \mathcal{D}'(\R^d)$.
\item[$(ii)$] $T_x(E) \subseteq E$ for all $x \in \R^d$.
\end{itemize}
A TIBD $E$ is called solid \cite[p.\ 311]{F-G} if  $E \subset L^1_{\loc}(\R^d)$ with continuous inclusion and 
$$
\forall f \in E \, \forall g \in L^1_{\loc}(\R^d) \, : \; |g| \leq |f| \mbox{ a.e.} \Longrightarrow g \in E \mbox{ and } \|g\|_E \leq \| f \|_E.
$$
\end{definition}
Let $E$ be a TIBD. Our first goal is to show that $(T_x)_{x \in \R^d}$ is a $C_0$-group on $E$. Observe that $T_x: E \rightarrow E$, $x \in \R^d$, is continuous, as follows from the closed graph theorem and the continuity of $T_x: \mathcal{D}'(\R^d) \rightarrow \mathcal{D}'(\R^d)$. Next, it is clear that $T_0 = \operatorname{id}$ and that $T_{x + y} = T_{x} \circ T_{y}$ for all $x,y \in \R^d$. We now show that 
\begin{equation}
\lim_{x \rightarrow 0}T_xe = e, \qquad e \in E.
\label{transl-C0}
\end{equation}
The weight function of the translation group on $E$ is defined  as
$$
\omega_E(x) := \| T_{-x} \|_{\mathcal{L}(E)}, \qquad x \in \R^d.
$$
The function $\omega_E$ is measurable (as $E$ is separable, which follows from $\mathcal{D}(\R^d) \hookrightarrow E$) and submultiplicative. These two properties imply that $\omega_E$ is exponentially bounded \cite[16.2.6] {Kuczma}, whence \eqref{transl-C0}  follows from the fact that $\mathcal{D}(\R^d)$ is dense in $E$. Summarizing, we have proven the next result.
\begin{proposition}\label{Czero}
Let $E$ be a TIBD. Then, $(T_x)_{x \in \R^d}$ is a $C_0$-group on $E$. 
\end{proposition}

\begin{remark}\label{general-def}
In \cite{D-Pi-V}, a  Banach space $E$ is called a TIBD if $E$ satisfies $(i)$, $(ii)$ and the additional property that the weight function $\omega_E$ is polynomially bounded. In the present framework, we can only conclude that $\omega_E$ is exponentially bounded. However, all results from \cite{D-Pi-V} remain valid upon replacing the spaces $\mathcal{S}(\R^d)$ and $\mathcal{S}'(\R^d)$ by $\mathcal{K}_1(\R^d)$ and $\mathcal{K}'_1(\R^d)$, respectively, where $\mathcal{K}_1(\R^d)$ stands for the space of exponentially decreasing smooth functions and $\mathcal{K}'_1(\R^d)$ stands for the space of exponentially bounded distributions \cite{Hasumi}.
 \end{remark}
\begin{definition}\label{Def-2}
Let $E$ be a TIBD. We define $\mathcal{D}_E$ as the space of all $e \in E$ such that $\partial^\alpha e\in E$ for all $\alpha \in \N^d$ (the derivatives should be interpreted in the sense of distributions). We endow $\mathcal{D}_E$ with the system of norms $\{ \| \, \cdot \, \|_{E,n} \, | \, n \in \N \}$, where 
$$
\| e \|_{E,n} := \max_{|\alpha| \leq n} \| \partial^\alpha e \|_E, \qquad e \in \mathcal{D}_E, n \in \N.
$$
Hence, $\mathcal{D}_E$  becomes a Fr\'echet space. 
\end{definition}

Let $E$ be a TIBD. In \cite[Prop.\ 7]{D-Pi-V}, it is shown that all elements $e$ of  $\mathcal{D}_E$ are smooth functions on $\R^d$ that satisfy
$$
\lim_{x \to \infty} \frac{\partial^\alpha e(x) }{\omega_E(-x)} = 0, \qquad \alpha \in \N^d.
$$
\begin{example}
The main examples of TIBD are the Lebesgue spaces $L^p(\R^d)$ ($1\leq p < \infty$), the space $C_0(\R^d)$ of continuous functions vanishing at infinity and their weighted variants. More precisely, let $\eta$ be a positive measurable function on $\R^d$ such that 
\begin{equation}
\sup_{t \in \R^d} \frac{\eta(\, \cdot + t )}{\eta(t)} \in L^\infty_{\operatorname{loc}}(\R^d).
\label{condition-eta}
\end{equation}
We define $L^p_\eta(\R^d)$ ($1 \leq p < \infty$) as the Banach space of all measurable functions $f$ on $\R^d$ such that $\| f \eta \|_{L^p} < \infty$. On the other hand, $C_{\eta,0}(\R^d)$ stands for the space of all continuous functions $f$ on $\R^d$ such that
$\lim_{x \to \infty} f(x)/\eta(x) = 0$; endowed with the norm  $\| \, \cdot / \eta \|_{L^\infty}$ it becomes a Banach space. The spaces $L^p_\eta(\R^d)$ and $C_{\eta,0}(\R^d)$ are clearly solid TIBD and we have that
$$
\mathcal{D}_{L^p_{\eta}} = \{ \varphi \in C^\infty(\R^d) \, | \, \partial^\alpha \varphi \in L^p_\eta(\R^d) \mbox{ for all } \alpha \in \N^d \}
$$
and
$$ 
\mathcal{D}_{C_{\eta,0}} = \dot{\mathcal{B}}_{\eta} =  \{ \varphi \in C^\infty(\R^d) \, | \, \partial^\alpha \varphi \in C_{\eta,0}(\R^d) \mbox{ for all } \alpha \in \N^d \}.
$$
 For $\eta \equiv 1$, these spaces were already considered by Schwartz \cite{Schwartz}.
 
 We would like to point out that $L^2(\R^d) \widehat{\otimes}_\pi  L^2(\R^d)$ and $L^2(\R^d) \widehat{\otimes}_\varepsilon  L^2(\R^d)$ are examples of TIBD consisting of locally integrable functions on $\R^{2d}$ that are different from all mixed Lebesgue spaces $L^{p,q}(\R^{2d})$ ($1\leq p,q \leq \infty$) \cite[Remark 3.11]{D-P-P-V}. In fact, the space $L^2(\R^d) \widehat{\otimes}_\pi  L^2(\R^d)$ is not even solid \cite[Remark 3.10]{D-P-P-V}.
\end{example}

\begin{theorem}\label{question}
Let $E$ be a TIBD. Then, $\mathcal{D}_E$ satisfies $(\Omega)$. In particular, $\mathcal{D}_E$ is quasinormable and distinguished. 
\end{theorem}
In view of Corollary \ref{Banach-omega}, Theorem \ref{question} is a direct consequence of the following lemma.
\begin{lemma}\label{smooth-TIB}
Let $E$ be a TIBD and consider the $C_0$-group $(T_x)_{x \in \R^d}$ on $E$ (cf.\ Proposition \ref{Czero}). Then, $E^\infty = \mathcal{D}_E$ and   $e^{(\alpha)} =  \partial^\alpha e$ for all $e \in\mathcal{D}_E$ and $\alpha \in \N^d$. 
\end{lemma}
\begin{proof}
We first show that $E^\infty \subseteq \mathcal{D}_E$  and $e^{(\alpha)} = \partial^\alpha e$ for all $e \in E^\infty$ and $\alpha \in \N^d$. It suffices to prove that, for all $e \in E^1$ and $1 \leq j \leq d$, $\partial_j e$ belongs to $E$ and $\gamma_{\partial_je} = \partial_j\gamma_e$, the result then follows by induction. For all $\varphi \in \mathcal{D}(\R^d)$ it holds that
$$
\partial_j\varphi = \lim_{h \to 0} \frac{T_{he_j}\varphi - \varphi}{h} \qquad \mbox{in } \mathcal{D}(\R^d).
$$
Hence,
$$
\langle T_x\partial_je, \varphi \rangle = -  \lim_{h \to 0} \left \langle T_xe,  \frac{T_{he_j}\varphi - \varphi}{h} \right \rangle = \lim_{h \to 0} \left \langle \frac{\gamma_e(x- he_j) - \gamma_e(x)}{-h}, \varphi \right \rangle = \langle  \partial_j\gamma_e(x), \varphi \rangle,
$$
for all $x \in \R^d$. The result now follows by setting $x = 0$ in the above equality. Next, we prove that  $\mathcal{D}_E \subseteq E^\infty$. Let $e \in \mathcal{D}_E$ be arbitrary.  By \cite[Appendice Lemme II]{Schwartz54} it is enough to show that $\gamma_e$ is scalarly smooth, i.e., $f_{e'} = \langle e', \gamma_e(\cdot) \rangle \in C^\infty(\R^d)$ for all $e' \in E'$. We claim that $\partial^\alpha f_{e'} = \langle e', \gamma_{\partial^\alpha e} (\cdot)\rangle$ in $\mathcal{D}'(\R^d)$ for all $\alpha \in \N^d$. This implies that $f_{e'}$ and all its derivatives  are continuous functions on $\R^d$, whence $f_{e'} \in C^\infty(\R^d)$. We now prove the claim. Let $\varphi \in \mathcal{D}(\R^d)$ be arbitrary. The equality
\begin{equation}
\int_{\R^d} \partial^\alpha\varphi(x)   T_xe  \dx  = (-1)^{|\alpha|} \int_{\R^d} \varphi(x)   T_x \partial^\alpha e  \dx
\label{cool}
\end{equation}
holds if we interpret the above integrals as $\mathcal{D}'(\R^d)$-valued integrals. Since both the functions $x \rightarrow \partial^\alpha\varphi(x)   T_xe$ and $x \rightarrow \varphi(x)   T_x \partial^\alpha e$ are compactly supported continuous $E$-valued functions, these integrals exist as $E$-valued integrals and we may conclude that the equality \eqref{cool} holds in $E$. Hence,
\begin{align*}
\langle \partial^{\alpha}f_{e'}, \varphi \rangle &= (-1)^{|\alpha|} \int_{\R^d} \partial^{\alpha}\varphi(x)  \langle e',\gamma_e(x) \rangle \dx\\ 
&=  (-1)^{|\alpha|} \left \langle e', \int_{\R^d} \partial^{\alpha}\varphi(x)   T_xe  \dx  \right \rangle \\  
&=   \left \langle e', \int_{\R^d} \varphi(x)   T_x \partial^{\alpha}e \dx  \right \rangle \\ 
&= \int_{\R^d} \varphi(x)  \langle e', \gamma_{\partial^{\alpha}e}(x) \rangle \dx.
\end{align*}
\end{proof}
To end this article, we address  the following problem.
\begin{problem} \label{problem-Montel} \cite[Remark 6]{D-Pi-V}  \emph{Do there exist TIBD $E$ such that $\mathcal{D}_E$ is Montel?}
\end{problem}
Schwartz already observed that the spaces $\mathcal{D}_{L^p}$ ($1 \leq p < \infty$) and $\dot{\mathcal{B}}$ are not Montel \cite[p.\ 200]{Schwartz}. Let $\eta$ be a positive measurable function on $\R^d$ satisfying \eqref{condition-eta}. Since $\mathcal{D}_{L^p_\eta} \cong \mathcal{D}_{L^p}$ ($1 \leq p < \infty$) and  $\dot{\mathcal{B}}_\eta \cong  \dot{\mathcal{B}}$ \cite[Remark 5]{D-Pi-V},  the spaces $\mathcal{D}_{L^p_\eta}$  and $\dot{\mathcal{B}}_\eta$ are also not Montel. Furthermore,  $\mathcal{D}_E$ is not Montel if $E$ is a TIBD such that $\omega_E$ is bounded on $\R^d$ \cite[Remark 6]{D-Pi-V}. Below we show that $\mathcal{D}_E$ is not Montel if $E$ is a solid TIBD. We  believe that  $\mathcal{D}_E$ is never Montel  but were unable to show this for general TIBD $E$. Observe that, by Theorem \ref{question}, it suffices to show that $\mathcal{D}_E$ is not Schwartz.
\begin{theorem}\label{non-montel-solid}
Let $E$ be  a solid TIBD. Then, $\mathcal{D}_E$ is not Montel.
\end{theorem}
The proof of Theorem \ref{non-montel-solid} is based on the well-known fact that all normed subspaces of a Montel lcHs are finite-dimensional.
Namely, following \cite[Def.\ 3.4]{F-G}, we associate to each  solid TIBD $E$ an infinite-dimensional Banach sequence space (extending the natural correspondence between $L^p$ and $l^p$, and $C_0$ and $c_0$) and then show that this sequence space is isomorphic to a  subspace of $\mathcal{D}_E$. We need some preparation.

We denote by $\chi_A$ the characteristic function of a set $A \subseteq \R^d$. Let $E$ be  a solid TIBD.  For $0 < r < 1/2$ we define $E_r(\Z^d)$ as the space of all  (multi-indexed) sequences $(a_j)_{j} \in \C^{\Z^d}$ such that
$$
\sum_{j \in \Z^d} a_j T_{-j}\chi_{B(0,r)} \in E.
$$
We endow $E_r(\Z^d)$ with the norm induced by $E$, that is,
$$
\| (a_j)_j \|_{E_r(\Z^d)} := \|  \sum_{j \in \Z^d} a_j T_{-j}\chi_{B(0,r)} \|_E, \qquad  (a_j)_j \in E_r(\Z^d).
$$
The solidness of $E$ implies that $E_r(\Z^d)$ is a solid sequence space, i.e.,
\begin{gather*}
\forall (a_j)_j \in E_r(\Z^d) \, \forall (b_j)_j \in \C^{\Z^d} \, : \; \\
|b_j| \leq |a_j| \mbox{ $\forall j \in \Z^d$} \Longrightarrow (b_j)_j \in E_r(\Z^d) \mbox{ and } \|(b_j)_j\|_{E_r(\Z^d)} \leq \|(a_j)_j\|_{E_r(\Z^d)}.
\end{gather*}
In the next lemma, we collect several useful facts about $E_r(\Z^d)$ (cf.\ \cite[Lemma 3.5]{F-G}).
\begin{lemma}\label{montel-lemma}
Let $E$ be  a solid TIBD and let $ 0 < r < 1/2$. Then,
\begin{itemize}
\item[$(i)$] $E_r(\Z^d)$ is a Banach space.
\item[$(ii)$] $E_r(\Z^d)$ is infinite-dimensional.
\item[$(iii)$] For all $f \in L^\infty(\R^d)$ with $\operatorname{supp} f \subseteq \overline{B}(0,r)$ and all $(a_j)_j \in E_r(\Z^d)$ it holds that
$$
 \sum_{j \in \Z^d} a_j T_{-j}f \in E \qquad \mbox{ and }  \qquad \|  \sum_{j \in \Z^d} a_j T_{-j}f \|_E \leq  \|f\|_{L^\infty} \|(a_j)_j\|_{E_r(\Z^d)}.
$$
\item[$(iv)$] For all $0 < R < 1/2$ it holds that $E_r(\Z^d) = E_{R}(\Z^d)$ with equivalent norms. 
\end{itemize}
\end{lemma}
\begin{proof}
$(i)$ Let $(a_n)_{n \in \N} = ((a_{n,j})_{j})_{n \in \N}$ be a Cauchy sequence in $E_r(\Z^d)$. Set 
$$
e_n = \sum_{j \in \Z^d} a_{n,j} T_{-j}\chi_{B(0,r)} \in E, \qquad n \in \N,
$$
and notice that $(e_n)_{n \in \N}$ is a Cauchy sequence in $E$. Hence, there is $e \in E$ such that $e_n \rightarrow e$ as $n \to \infty$. On the other hand, as $E_r(\Z^d) \subset \C^{\Z^d}$ with continuous inclusion, there is $a = (a_j)_{j} \in \C^{\Z^d}$ such that $a_{n,j} \rightarrow a_j$ as $n \to \infty$ for all $j \in \Z^d$. Since
$$
e = \sum_{j \in \Z^d} a_{j} T_{-j}\chi_{B(0,r)} \mbox{ in } \mathcal{D}'(\R^d),
$$
we obtain that $a \in E_r(\Z^d)$. Finally, $a_n \rightarrow a$ in $E_r(\Z^d)$ as $n \to \infty$ because $e_n \rightarrow e$ in $E$ as $n \to \infty$.

\noindent$(ii)$ Since $\mathcal{D}(\R^d) \subset E$ and $E$ is solid, the characteristic function of every compact set in $\R^d$ belongs to $E$. Hence, all unit vectors belong to $E_r(\Z^d)$.

\noindent$(iii)$ This follows from the solidness of $E$.

\noindent$(iv)$ We may assume without loss of generality that $R  \leq r$. The continuous inclusion $E_r(\Z^d) \subseteq E_R(\Z^d)$ follows from $(iii)$ with $f = \chi_R$. For the converse, choose a finite set of points $\{x_1, \ldots, x_n\} \subset B(0,r)$ such that
$$
B(0,r) \subseteq \bigcup_{m = 1}^n B(x_m,R)
$$
and, thus,
$$
\chi_{B(0,r)} \leq \sum_{m=1}^n T_{-x_m} \chi_{B(0,R)}.
$$
Now let $(a_j)_j \in E_R(\Z^d)$ be arbitrary. Then,
$$
| \sum_{j \in \Z^d} a_j T_{-j} \chi_{B(0,r)} | \leq \sum_{j \in \Z^d} \sum_{m=1}^n |a_j| T_{-j-x_m} \chi_{B(0,R)} \leq  |\sum_{m=1}^n T_{-x_m}|\sum_{j \in \Z^d} a_j T_{-j} \chi_{B(0,R)} ||. 
$$
Since $E$ is solid and translation-invariant, we obtain that $(a_j)_j \in E_r(\Z^d)$ and that
$$
\| (a_j)_j \|_{E_r(\Z^d)} \leq  \| \sum_{m=1}^n T_{-x_m}|\sum_{j \in \Z^d} a_j T_{-j} \chi_{B(0,R)} | \|_E \leq  \sum_{m=1}^n \omega_E(x_m) \| (a_j)_j \|_{E_R(\Z^d)}.
$$
\end{proof}
Finally, we show that $E_r(\Z^d)$ is isomorphic to a subspace of $\mathcal{D}_E$. As explained above, taking Lemma \ref{montel-lemma}$(ii)$ into account,  this completes the proof of Theorem \ref{non-montel-solid}. 
\begin{proposition}
Let $E$ be  a solid TIBD. Fix $ 0 < r < 1/2$ and $\varphi \in \mathcal{D}(\R^d)$ with $\operatorname{supp} \varphi \subset B(0,r)$  such that $\varphi \equiv 1$ in a neighbourhood of $0$. Then,
$$
\iota: E_r(\Z^d) \rightarrow \mathcal{D}_E: (a_j)_j \rightarrow  \sum_{j \in \Z^d} a_j T_{-j}\varphi
$$
is a topological embedding.
\end{proposition}
\begin{proof}
The mapping $\iota$ is well-defined and continuous by Lemma \ref{montel-lemma}$(iii)$. We now show that $\iota$ is a strict morphism. Choose $0 < R < r$ such that $\varphi \equiv 1$ on $B(0,R)$. Then,
$$
| \sum_{j \in \Z^d} a_j T_{-j} \chi_{B(0,R)} | \leq  |\sum_{j \in \Z^d} a_j T_{-j}\varphi |, \qquad (a_j)_j \in E_r(\Z^d).
$$
Since $E$ is solid, we obtain that
$$
\| (a_j)_{j}\|_{E_R(\Z^d)} \leq \|\iota((a_j)_{j}) \|_E, \qquad (a_j)_j \in E_r(\Z^d).
$$
The result now follows from Lemma \ref{montel-lemma}$(iv)$.
\end{proof}

\end{document}